\documentclass[12pt,a4paper]{amsproc}
\usepackage{amssymb}
\usepackage{setspace} 
\usepackage[pdftex]{graphicx}
\usepackage{caption}
\usepackage{ mathrsfs }
\usepackage{enumerate}

\def\DJ{\leavevmode\setbox0=\hbox{D}\kern0pt\rlap
 {\kern.04em\raise.188\ht0\hbox{-}}D}
\def\dj{\leavevmode
 \setbox0=\hbox{d}\kern0pt\rlap{\kern.215em\raise.46\ht0\hbox{-}}d}

\theoremstyle{plain}
 \newtheorem{thm}{Theorem}
 
 \newtheorem{lem}{Lemma}[section]
 
\theoremstyle{definition}
 
 \newtheorem{rem}{Remark}[section]

\numberwithin{equation}{section}
\numberwithin{equation}{section}

\title[Boyd's conjecture]{BOYD'S CONJECTURE}



\author[D. Stankov]{\bfseries Dragan Stankov}

\address{
Katedra Matematike RGF-a  \\ 
Universitet u Beogradu   \\ 
11000 Beograd, \DJ u\v sina 7\\
Serbia}
\email{dstankov@rgf.bg.ac.rs, draganstankov@yahoo.com}

\begin{document}

\setcounter{page}{1}

\begin{abstract}
\normalsize 

We determine the limit of the rate $\frac{\nu_{n,a}}{n}$ between the number $\nu_{n,a}$ of roots of the trinomial $x^n-ax-1$, $a\in (0,2]$, which are greater than 1 in modulus, and degree $n$. The analogue of Boyd's Conjecture (C) for Perron numbers is a consequence of the limit, under the assumption that the conjecture of Lind-Boyd is valid. The product of these $\nu_{n,a}$ roots has also a limit when $n\to\infty$. The explicit expression of the limit by an integral is presented. The computing of the rate and the product for $n=100,150$ as well as of its limits is presented.

\end{abstract}

\maketitle

\section{Introduction}

Let $\alpha$ be an algebraic integer of degree $n$, whose conjugates are $\alpha_1 = \alpha,
\alpha_2,\ldots,\alpha_n$ and
$p=b_0x^n +b_1x^{n-1} +\cdots+b_{n-1}x +b_n$, with $b_0=1$,
its minimal polynomial. A Perron number, which was defined by Lind \cite{Lin1}, is
a real algebraic integer $\alpha$ of degree $n \geq 2$ such that $\alpha > |\alpha_i|$, $i=2,\ldots n$. Any Pisot number or Salem number is a Perron number.

Lind \cite{Lin1} conjectured that the smallest Perron number of degree $n \geq 2$ should
have minimal polynomial $x^n-x-1$. Boyd \cite{Boy1} has computed all smallest
Perron numbers of degree $n \leq 12$, and found out that Lind's conjecture is true if
$n = 2,3,4,6,7,8,10$, but false if $n > 3$ and $n\equiv 3$ or $n\equiv 5$ (mod 6). So in \cite{Boy1}, we have

Conjecture (Lind-Boyd). The smallest Perron number $\alpha$ of degree $n > 2$ has minimal
polynomial

$x^n-x-1$ if $n \not \equiv 3,5$ (mod 6),

$(x^{n+2}-x^4-1)/(x^2-x+1)$ if n $\equiv 3$ (mod 6),

$(x^{n+2}-x^2-1)/(x^2-x+1)$ if n $\equiv 5$ (mod 6).

Wu \cite{Wu} gave all Perron numbers of degree $13 \le n \le 24$ with $α\alpha \leq (2 +
1/n)^{1/n}$ and their minimal polynomials, and verified that all the smallest Perron
numbers of degree $13 \le n \le 24$ satisfy the conjecture of Lind-Boyd.

We denote, as usual, by
\[\overline{|\alpha|}=\max_{1\le i\le n}|\alpha_i|\]
the \textit{house} of $\alpha$. We define $m(n)$ to be the minimum of the houses of the algebraic integers $\alpha$ of degree $n$ which are not a root of unity. If $\nu_{n,a}$ is the number of roots $\alpha_i$, satisfying
$|\alpha_i|>1$, then Boyd \cite{Boy1} noticed in his Conjecture (C) that $\nu_{n,a} \sim \frac{2}{3} n$. First of all we shall formulate the analogue of Conjecture (C) for Perron numbers:

Conjecture (CP). If $\nu_{n,a}$ is the number of conjugates $\alpha_i$ of the smallest Perron number of degree $n > 2$, satisfying
$|\alpha_i|>1$, then $\nu_{n,a} \sim \frac{2}{3} n$.

The aim of this paper is to show that if the conjecture of Lind-Boyd is true then Conjecture (CP) should also be true. Actually we shall prove

\begin{thm}\label{cha:nuApprox}
The rate $\frac{\nu_{n,a}}{n}$ between the number $\nu_{n,a}$ of roots of the trinomial $x^n-ax-1$, $a\in (0,2]$, which are greater than 1 in modulus, and degree $n$, tends to $\frac{1}{\pi}\arccos(-\frac{a}{2})$, $n\to \infty$.
\end{thm}

Let $\textrm{M}(\alpha) =|b_0|\prod_{i=1}^n \max(|\alpha_i|, 1)$ denote the Mahler measure of $\alpha$ (and of p).
Jensen's formula which states that
\[\int_0^1\ln|p(e^{2\pi i\theta})|d\theta=\ln|b_0|+\sum_{i=1}^{n}\ln \max(|\alpha_i|,1), \]
leads us to the following result \[ \textrm{M}(p)=\exp \left(\int_0^1 \ln |p(e^{2\pi i\theta})|d\theta\right).\]
Thus Mahler measure could be extended to polynomials in several variables
\[ \textrm{M}(p)=\exp \left(\int_0^1\cdots\int_0^1 \ln |p(e^{2\pi i\theta_1},\ldots,e^{2\pi i\theta_m})|d\theta_1\cdots d\theta_m \right).\]

Lehmer \cite{Leh} asked:
(L) Does there exist a constant $c_0 > 1$ so that $\textrm{M}(\alpha) > c_0$ for all $\alpha$ not roots of
unity? The smallest known Mahler measure (greater than 1) is for a root $\alpha$ of the polynomial
$x^{10}+x^9-x^7-x^6-x^5-x^4-x^3+x+1$ for which the Mahler measure is the Salem number
$\textrm{M}(\alpha)=1.176280818\dots $ It is widely believed that this number represents the minimal value in Lehmer's conjecture.

Let $\alpha$ be the Perron number having minimal polynomial $x^n-x-1$. Taking $a=1$ in the formula of the following theorem we are able to calculate the limit of the Mahler measure of $\alpha$, $n\to\infty$.
\begin{thm}\label{cha:Product}
The product of roots of the trinomial $x^n-ax-1$, $a\in (0,2]$, which are greater than 1 in modulus, tends to \[\exp\left(\frac{1}{2\pi}\int_0^{\arccos(-\frac{a}{2})}\ln(1+a^2+2a\cos t)dt \right),\;\; n\to \infty.\]
\end{thm}

\section{Lima\c{c}on}

If we represent roots of the trinomial $x^n-x-1$ in the complex plane we can notice that all of them lay on a heart shape curve.
\begin{figure}[!htbp]
\caption{Roots of polynomials $x^{24}-x-1$, $x^{24}-1$}
\begin{center}
\includegraphics {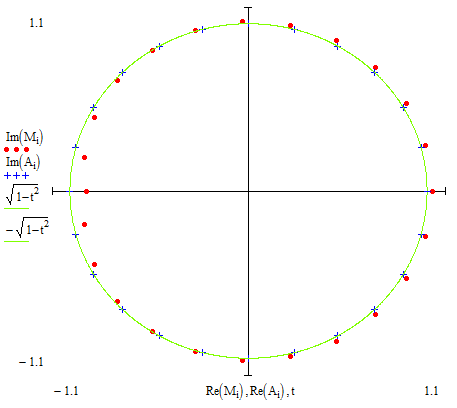}
\end{center}
\end{figure}

\begin{lem}\label{sec:Limacon}
Let $\alpha=\rho(\cos\varphi+i\sin\varphi)$ be a root of the trinomial $x^n-ax-1$. Then $\rho$, $\varphi$ satisfy equations
\begin{equation}\label{Trig1}
\rho^{2n}=a^2\rho^2+2a\rho\cos\varphi+1
\end{equation}
\begin{equation}\label{Trig2}
a\rho\sin(n-1)\varphi+\sin n\varphi=0.
\end{equation}
\end{lem}
\begin{proof}
Since $\alpha$ satisfies the equation $x^n=ax+1$ we have
\begin{equation}\label{Trig3}
\rho^n(\cos n\varphi+i\sin n\varphi)=a\rho(\cos\varphi+i\sin\varphi)+1.
\end{equation}
If we determine the square of the modulus of both sides of the equation \eqref{Trig3} we get
\begin{eqnarray*}
\rho^{2n} &= &(a\rho\cos\varphi+1)^2+(a\rho\sin\varphi)^2\\
     &=&a^2\rho^2+2a\rho\cos\varphi+1\\
\end{eqnarray*}
If we separate real and imaginary part of the equation \eqref{Trig3} we get the system of two equations:
\[ \rho^n\cos n\varphi=a\rho\cos\varphi+1\]
\begin{equation}\label{Trig5}
\rho^n\sin n\varphi=a\rho\sin\varphi.
\end{equation}
If we multiply first of them with $\sin n\varphi$ and second with $\cos n\varphi$ we get the equations with equal left sides, so the difference of its right sides must be equal to 0:
\begin{equation}\label{Trig4}
(a\rho\cos\varphi+1)\sin n\varphi -a\rho\sin\varphi\cos n\varphi=0.
\end{equation}
Now it is obvious that
\[a\rho(\cos\varphi\sin n\varphi-\sin\varphi\cos n\varphi)+\sin n\varphi=0.\]
Using the well known formula of sine of difference of two angles we finally get the equation \eqref{Trig2}.
\end{proof}

\begin{lem}\label{sec:RealZero}
The trinomial $P(x)=x^n-ax-1$ has a unique positive real root $\beta > 1$  if $a>0$. Furthermore $\beta$ converge to 1 above when $n$ tends to infinity.
\end{lem}
\begin{proof} By Descartes' rule of signs $P(x)$ has exactly one positive root. If $x=1$ then $P(x)=-a<0$. If $x=1+\sqrt[n-1]{a}$ then $P(x)>1+\sqrt[n-1]{a^n}-a\sqrt[n-1]{a}-1=0$. So there is a unique positive real root $\beta > 1$ on $(1,1+\sqrt[n-1]{a})$ which converge to 1 above when $n$ tends to infinity.
\end{proof}
\begin{lem}\label{sec:RealZero1}
All roots of the trinomial $P(x)=x^n-ax-1$, $a>0$ lie in the circle $|x|\le\beta$ where $\beta > 1$ is its unique positive real root.
\end{lem}
\begin{proof} The existence of $\beta$ is proved in the previous lemma. If $|x|>\beta$ then $P(|x|)=|x|^n-a|x|-1>0$. Since $|P(x)|\ge|x|^n-a|x|-1=P(|x|)>0$ we conclude that all roots of the trinomial $P(x)$ must be in the circle $|x|\le\beta$.
\end{proof}

\begin{lem}\label{sec:RealZero2}
If $a>0$ the polynomial $P_1(x)=1-x^n-ax$ is decreasing on $(0,\infty)$ and has unique positive real root  $\alpha < 1$. Then there is no root of the trinomial $P(x)=x^n-ax-1$ in the circle $|x|<\alpha$.
\end{lem}
\begin{proof} If $|x|<\alpha$ then $P_1(|x|)>0$. Since $|P(x)|\ge 1-|x|^n-a|x|=P_1(|x|)>0$ we conclude that all roots of the trinomial $P(x)$ must be out of the circle $|x|<\alpha$.
\end{proof}

\begin{lem}\label{sec:RealZero3}
If $a\in(1,2]$, there is $n_0$ such that if $n>n_0$ the polynomial $P_2(x)=ax-x^n-1$ has exactly two real roots $\gamma_1,\gamma_2 $ on $(0,1]$. Then there is exactly one root of the trinomial $P(x)=x^n-ax-1$ in the circle $K:|x|=\gamma$ for any $\gamma $ that satisfies $ \gamma_1<\gamma<\gamma_2$.
\end{lem}
\begin{proof} If $|x|=\gamma$ then $P_2(|x|)>0$. Since $|-ax|-|x^n-1|\ge |-a||x|-|x|^n-1=P_2(|x|)>0$ we conclude that $|x^n-1|<|-ax|$ on $K$. Now we can use Rouch\'{e}'s theorem and conclude that $-ax$ and $P(x)=x^n-ax-1$  have the same number of zeros inside $K$.
\end{proof}
\begin{rem} If n is odd then $P(-\gamma_1)=P_2(\gamma_1)=0$ which means that $-\gamma_1$ is this exactly one root of the trinomial $P(x)$ in $K$. Lemma \ref{sec:RealZero3} allows us to conclude that there are no roots of $P(x)$ in the ring $\{x:\gamma_1<|x|<\gamma_2\}$.
\end{rem}
\begin{lem}\label{sec:ModuoZero}
If $a\in(0,1]$ then any root of the trinomial $P_n(x)=x^n-ax-1$ can be arbitrary close to 1 in modulus when $n\to\infty$.
\end{lem}
\begin{proof} If we determine $\beta$ for $P_n(x)$ as in Lemma \ref{sec:RealZero} and $\alpha$ as in Lemma \ref{sec:RealZero2} we can see that $\alpha$ converge to 1 below. We have already proven in Lemma \ref{sec:RealZero} that $\beta$ converge to 1 above as $n\to\infty$. Since moduli of all roots are between $\alpha$ and $\beta$ the statement is shoved.
\end{proof}

\begin{lem}\label{sec:ModuoZero1}
If $a\in(1,2]$ then there is a real root of the trinomial $P_n(x)=x^n-ax-1$ arbitrary close to $-\frac{1}{a}$. All other roots of the trinomial converge to 1 in modulus when $n\to\infty$.
\end{lem}
\begin{proof} The equation $P_n(x)=0$ is equivalent with $x^n=ax+1$. We can see, from the graphic representation of these two functions, that they have an intersection point, corresponding to real root $r_1$, arbitrary close to $(-\frac{1}{a},0)$ when $n\to\infty$. If we determine $\beta$ for $P_n(x)$ as in Lemma \ref{sec:RealZero1} and $\gamma_2$ as in Lemma \ref{sec:RealZero3} we can see that $\gamma_2$ converge to 1 below and $\beta$ converge to 1 above as $n\to\infty$. Since moduli of all roots except $r_1$ are between $\gamma_2$ and $\beta$ the claim is proved.
\end{proof}
\begin{rem} If we use the substitution $\rho^{2n}=R$ and the fact that $\rho\sim 1$ then \eqref{Trig1} might be approximated with the equation $R=c+d\cos\varphi$. It represents a curve known as a lima\c{c}on of (\' Etienne) Pascal (father of Blaise Pascal). A much more known curve, the cardioid $R=2b(1+\cos\varphi)$, is a special case of a lima\c{c}on.
\end{rem}
Now we know that all roots, except eventually one, are in the ring arbitrary close to the unit circle. It is very important to determine their position in the complex plane more precisely. If $\alpha_1=\rho_1 e^{i\varphi_1}$, $\alpha_2=\rho_2 e^{i\varphi_2}$ are two roots of the trinomial $P_n(x)=x^n-ax-1$ such that $0<\varphi_1<\varphi_2<\pi$ then $\rho_1>\rho_2$ i.e. absolute value of a root decrease as its argument increase from 0 to $\pi$. This fact enables us a simple method to split roots which are in, from those that are out of the unit circle.

\begin{lem}\label{sec:ImplFunc}
If $a\in(0,2)$ the equation \eqref{Trig1} defines implicitly the function $\rho=\rho(\varphi)$ which is an decreasing function on $[0,\pi)$ if $n>n_0$.
\end{lem}
\begin{proof} If we formally solve the equation in $\cos\varphi$ we get
\[\cos\varphi=\frac{\rho^{2n}-a^2{\rho}^2-1}{2a\rho}.\]
Let us denote the function on the right with $g(\rho)$. The previous two lemmas show that we need to analyse $g(\rho)$ in a neighborhood of 1. Since $\cos\varphi=g(\rho)$ we are interested only for $\rho$ such that $g(\rho)\in[-1,1]$. It is easy to show that $g(\beta)=1$ and  $g(1)=\frac{-a}{2}\in(-1,0)$. The first derivative of $g$ is
\[g'(\rho)=\frac{1}{2a}\frac{(2n-1)\rho^{2n}-a^2\rho^2+1}{\rho^2}.\]
If $\rho\in[1,\beta]$ then $g'(\rho)>\frac{1}{2a}\frac{(2n-1)-a^2\beta^2+1}{\beta^2}$ is greater than 0 if $n>n_1$. If $\rho<1$ and $a\rho>1$ it is obvious that $g'(\rho)$ could be negative. But if $g(\rho)=\cos\varphi\in(-1,1]$ then we intend to show that there is $n_2$ such that $g'(\rho)>0$ for all $n>n_2$.
Using the equation \eqref{Trig1} we have:
\begin{eqnarray*}
g'(\rho)&=&\frac{1}{2a}\frac{(2n-1)\rho^{2n}-a^2\rho^2+1}{\rho^2}   \\
     &=&\frac{1}{2a}\frac{(2n-1)(a^2\rho^2+2a\rho\cos\varphi+1)-a^2\rho^2+1}{\rho^2}\\
     &=&\frac{1}{2a}\frac{(2n-2)a^2\rho^2+2(2n-1)a\rho\cos\varphi+2n}{\rho^2}\\
     &=&\frac{1}{2a\rho^2}\left(\left(\sqrt{2n-2}a\rho+\frac{(2n-1)}{\sqrt{2n-2}}\cos\varphi\right)^2-\frac{(2n-1)^2}{2n-2}\cos^2\varphi+2n\right)\\
     &\ge&\frac{1}{2a\rho^2}\left(-\frac{(2n-1)^2}{2n-2}\cos^2\varphi+2n\right)\\
     &=&\frac{1}{2a\rho^2}\left(-\frac{m^2}{m-1}\cos^2\varphi+(m+1)\right),\;(m=2n-1)\\
     &=&\frac{1}{2a\rho^2}\frac{m^2(1-\cos^2\varphi)-1}{m-1}\\
     &=&\frac{1}{2a\rho^2}\frac{(2n-1)^2(1-\cos^2\varphi)-1}{2n-2}\\
     &>&0
 \end{eqnarray*}
when $n>n_2$, $\cos\varphi\ne -1$. If we take $n_0=\max(n_1,n_2)$ then $g(\rho)$ increase  and it is continuous, when $\rho$ is close to 1, thus there is $\beta_0<1$ such that $g:(\beta_0,\beta]\rightarrow(-1,1]$. Now we conclude that $\varphi=\arccos(g(\rho))$ decrease on $(\beta_0,\beta]$, for that reason it has an inverse function $\rho=\rho(\varphi)$ which is decreasing too and such that $\rho(\varphi):(0,\pi]\rightarrow(\beta_0,\beta]$.
\end{proof}

\begin{rem}\label{sec:RRealZero2}
It should be expected that $g'(\rho)$ could be negative if $g(\rho)<-1$ because for $a>1$, $n$ odd there might be two negative roots of $P(x)$ which correspond with $g(\rho)=\cos\varphi=-1$: first in a neighborhood of $\frac{-1}{a}$, second in a neighborhood of -1. If $\rho$ is between absolute values of these two roots $g(\rho)$ could not be monotonic.
\end{rem}

There is another property of roots of the trinomial $P(x)$ observed as points of the curve $\rho=\rho(\varphi)$: the adjacent roots on the curve are approximately equispaced in $\varphi$.

\begin{lem}\label{sec:RootPhi}
There is a partition of the interval $[0,(\lceil \frac{n}{2} \rceil -1)\frac{2\pi}{n}]$ on $\lceil \frac{n}{2} \rceil -1$ subintervals of the equal length $\frac{2\pi}{n}$, each of them contain exactly one $\varphi_j$ such that $\rho_j(\cos\varphi_j+i\sin\varphi_j)$ is a root of $P(x)$.
\end{lem}
\begin{proof}
A root of $P(x)$ satisfies \eqref{Trig4} which gives
\[\frac{a\rho\cos\varphi+1}{a\rho\sin\varphi}=\cot n\varphi\]
Since $\rho\approx 1$ the function on the left, let call it $R(\varphi)$, is approximately a constant on $I_k=(\frac{k\pi}{n},\frac{(k+1)\pi}{n})$, $k=1,2,\ldots,n-2$. The graph of $\cot n\varphi$ consists of parallel equispaced cotangents branches. We conclude that there is exactly one intersection point on $I_k$. Since $\rho>0$, if we bring to mind \eqref{Trig5}, it is obvious that $\sin n\varphi$ must be positive. So only intersection points on $(\frac{2m\pi}{n},\frac{(2m+1)\pi}{n})$ correspond to roots of the trinomial. Finally we can take the partition $[\frac{2m\pi}{n},\frac{(2m+2)\pi}{n}]$, $2m+2< n$.
\end{proof}

\begin{rem}\label{sec:Schinzel}
A. Schinzel (personal communication, January 14, 2014) suggested the following question to be explored: "does Lemma \ref{sec:RootPhi} follow from Erd{\H o}s-Tur\'{a}n estimate for the number of zeros of a given polynomial lying in a given angle". To show that the answer is affirmative let us, at first, remind the deep result of Erd{\H o}s and Tur\'{a}n.

\begin{thm}\label{sec:ErdTur}(Erd{\H o}s-Tur\'{a}n \cite{ET}) If the roots of the polynomial
\[P(z)=a_0+a_1z+\ldots+a_nz^n\]
are denoted by $z_{\nu}=r_{\nu}e^{i\varphi_{\nu}}$, $\nu=1,2,\ldots,n$ then for every $0\le\alpha<\beta\le 2\pi$ we have
\[\left|\sum_{\nu:\alpha\le\varphi_{\nu}\le\beta}1-\frac{\beta-\alpha}{2\pi}n\right|<16\sqrt{n\ln\frac{|a_0|+\cdots+|a_n|}{\sqrt{|a_0a_n|}}}=16\sqrt{n\ln R}\]
\end{thm}
For the trinomial $P(z)=z^n-az-1$ we can see that $R=2+|a|$, so it does not depend on $n$. If we divide both sides of the inequality by $n$ we obtain
\[\left|\frac{1}{n}\sum_{\nu:\alpha\le\varphi_{\nu}\le\beta}1-\frac{\beta-\alpha}{2\pi}\right|<16\sqrt{\frac{\ln (2+|a|)}{n}}\to 0,\;\;n\to\infty.\]
Thus we can conclude that arguments of roots of the trinomial are uniformly distributed on $[0,2\pi]$ as $n$ tends to infinity.
\end{rem}

\section{Proofs of the theorems}

Using all these lemmas, from the previous section, we are able now to prove Theorem \ref{cha:nuApprox}:
\begin{proof} If we set $\rho=1$ in the equation \eqref{Trig1}
\[0=a(a+2\cos\varphi)\]
we can solve it in $ \varphi$: $ \varphi=\arccos(-\frac{a}{2})$. It is proved in Lemma \ref{sec:ImplFunc} that the function $\rho=\rho(\varphi)$ is an decreasing function on $[0,\pi)$, so the modulus of a root $\alpha_j=\rho_j(\cos\varphi_j+i\sin\varphi_j)$ is greater than 1 if $\varphi_j\in[0,\arccos(-\frac{a}{2}))$. It follows from Lemma \ref{sec:RootPhi} that $\alpha_j$ are approximately uniformly distributed on $[0,\pi)$. We finally conclude that the rate $\frac{\nu_{n,a}}{n}$ between the number $\nu_{n,a}$ of roots of the trinomial $x^n-ax-1$, $a\in (0,2]$, which are greater than 1 in modulus, and degree $n$, tends to the rate between the length of the interval $[0,\arccos(-\frac{a}{2}))$ and the interval $[0,\pi)$ that is $\frac{1}{\pi}\arccos(-\frac{a}{2})$, $n\to \infty$.
\end{proof}

Theorem \ref{cha:Product} remains to be proved:
\begin{proof}
\begin{eqnarray*}
\prod_{|\alpha_i|>1}|\alpha_i|&=&\exp(\sum_{|\alpha_i|>1}\ln|\alpha_i|)   \\
  &=&\exp(\sum_{\rho_i>1}\ln\rho_i) \;\;\;  (\textrm{using} \;\eqref{Trig1})\\
  &=&\exp(\sum_{\rho_i>1}\frac{1}{2n}\ln(a^2{\rho_i}^2+2a\rho_i\cos\varphi_i+1))\;\;\;(*)\\
  &=&\exp(\sum_{0\le\varphi_i<\arccos(-\frac{a}{2})}\frac{2}{2n}\ln(a^2{\rho_i}^2+2a\rho_i\cos\varphi_i+1))\\
  &=&\exp(\frac{1}{2\pi}\sum_{0\le\varphi_i<\arccos(-\frac{a}{2})}\frac{2\pi}{n}\ln(a^2{\rho_i}^2+2a\rho_i\cos\varphi_i+1))\\
 \end{eqnarray*}
We use in (*) Theorem \ref{cha:nuApprox} and the well known theorem: if $\alpha$ is a root of a polynomial with real coefficients then $\bar{\alpha}$ is also its root. Since \[\ln(a^2+2a\cos\varphi_i+1)\le\ln(a^2{\rho_i}^2+2a\rho_i\cos\varphi_i+1)\le\ln[{\beta}^2(a^2+2a\cos\varphi_i+1)]\]
where $\beta$ is determined as in lemma \ref{sec:RealZero}, we conclude that
\[\exp(\frac{1}{2\pi}\sum_{0\le\varphi_i<\arccos(-\frac{a}{2})}\frac{2\pi}{n}\ln(a^2+2a\cos\varphi_i+1))
\le\prod_{|\alpha_i|>1}|\alpha_i|.\]
On the other hand
\begin{eqnarray*}
\prod_{|\alpha_i|>1}|\alpha_i|&\le&\exp(\frac{1}{2\pi}\sum_{0\le\varphi_i<\arccos(-\frac{a}{2})}\frac{2\pi}{n}(\ln\beta^2+\ln(a^2+2a\cos\varphi_i+1)))\\
    &\le&\exp(\ln\beta^2+\frac{1}{2\pi}\sum_{0\le\varphi_i<\arccos(-\frac{a}{2})}\frac{2\pi}{n}\ln(a^2+2a\cos\varphi_i+1))\\
    &=&\beta^2\exp(\frac{1}{2\pi}\sum_{0\le\varphi_i<\arccos(-\frac{a}{2})}\frac{2\pi}{n}\ln(a^2+2a\cos\varphi_i+1))\\
 \end{eqnarray*}
Since the function $\ln(a^2+2a\cos t+1)$ is Riemann-integrable the sum
\[\sum_{0\le\varphi_i<\arccos(-\frac{a}{2})}\frac{2\pi}{n}\ln(a^2+2a\cos\varphi_i+1)\]
tends to \[\int_0^{\arccos(-\frac{a}{2})}\ln(a^2+2a\cos t+1)dt ,\]
when $n\to\infty$, as an integral sum with the partition defined in lemma \ref{sec:RootPhi}.
If we bring to mind that $\beta\to 1$, $n\to\infty$ the theorem is proved.
\end{proof}

\begin{rem}\label{sec:Dubickas}
A. Dubickas (personal communication, January 14, 2014) noted that Theorem \ref{cha:Product} follows immediately from the fact that the (logarithmic) Mahler
measure $m(x^n-ax-1)$ tends to $m(y-ax+1)$ as $n$ tends to
infinity
where the logarithmic Mahler measures $m(p)=\ln M(p)$.
(The following limit formula is valid for a two variable polynomial $p$ \cite{Boy3}: 
$m(p(x,x^n))\to m(p(x,y))$, as n tends to $\infty$.) Using Maillot formula \cite{Ma} for $m(ax+by+c)$ which is generalized in \cite{Van} one could prove that
\begin{equation}\label{Maillot}
m(y-ax+1)=\frac{1}{\pi}(\textrm{Im}(\textrm{L}i_2(z))+\arg z\ln a)
\end{equation}
where $0<a\le 2$, $z=e^{2\arcsin(a/2)i}$ and dilogarithm $\textrm{L}i_2(z)=\sum_{k=1}^{\infty}\frac{z^k}{k^2}$.
\end{rem}

\section{Conjecture (CP)}

At last, we have to confirm Conjecture (CP) under the assumption that the conjecture of Lind-Boyd is valid. As we have seen, there are three cases in the conjecture of Lind-Boyd. In the first case Conjecture (CP) follows immediately as a corollary from Theorem \ref{cha:nuApprox}: if we take $a=1$ the rate is $\frac{2}{3}$. In the third case, if $n$ is even, we can use the substitution $x^2=t$ and refer to the first case. Likewise, in the second case, if $4|n$, we can use the substitution $x^4=t$ and refer to the first case again. It could be shown, using Theorem \ref{cha:nuApprox}, that in the other two cases Conjecture (CP) is also valid.

If we represent roots of trinomial $x^n-ax^2-1$ ($x^n-ax^4-1$) in the complex plane we can notice that all of them lay on a heart shape curve with two (four) cusps. 
\begin{figure}[!htbp]
\caption{Roots of polynomials $x^{24}-x^2-1$, $x^{24}-1$}
\begin{center}
\includegraphics {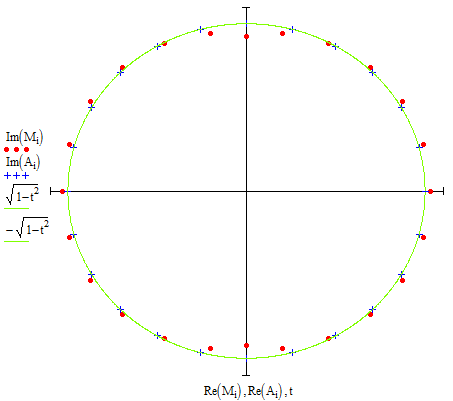}
\end{center}
\end{figure}
\begin{figure}[!htbp]
\caption{Roots of polynomials $x^{24}-x^4-1$, $x^{24}-1$}
\begin{center}
\includegraphics {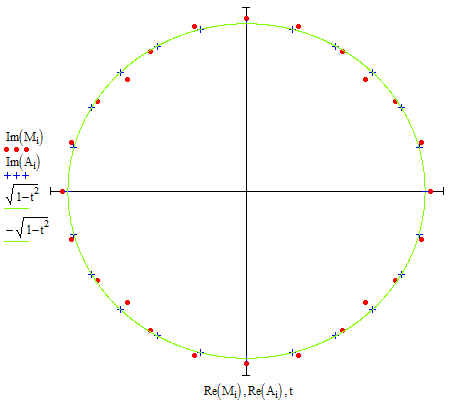}
\end{center}
\end{figure}
Using the same methods as in Chapter 2 we can show that 
\begin{enumerate}
                \item moduli of roots decrease while its arguments increase on $[0,\pi/2)$ ($[0,\pi/4)$),
                \item arguments of roots are uniformly distributed on $[0,2\pi)$,
                \item modulus of a root tends to 1 as $n$ tends to infinity,
                \item a root $\alpha= e^{i\varphi}$ lay on the curve $C_2:\rho^{2n}=a^2\rho^4+2a\rho^2\cos2\varphi+1$ ($C_4:\rho^{2n}=a^2\rho^8+2a\rho^4\cos4\varphi+1$),
                \item the subset of $[0,\pi/2]$ ($[0,\pi/4]$) on which $C_2$ ($C_4$) is out of the unit circle tends to $[0,\arccos\frac{-a}{2}\frac{\pi}{2})$ ($[0,\arccos\frac{-a}{2}\frac{\pi}{4})$).
\end{enumerate}

\begin{rem}\label{sec:Hare}
K. Hare (personal communication, January 23, 2014) asked:
"The Mahler measure result still holds for $a = 0$, although the $\nu$ result does not, as all
roots are on the unit circle.  What happens if $a < 0$, or $a >2$, or
$a$ complex with $0 < |a| \le 2$ or $|a| >2$ ?". 

If $|a|>2$ then there is exactly one root $\alpha_1$ inside the unit circle, of the trinomial $T_{n,a}=x^n-ax-1$ which should be close to $-1/a$. Thus, by the Vieta's formula, the Mahler measure of $T_{n,a}$ tends to $|a|$, $n\to \infty$. The same result we can get using the Maillot formula \cite{Van}. 

If we assume that $a=|a|e^{i\theta}$, $0<|a|\le 2$, $\theta=\frac{p}{q}2\pi$, $p,q\in \textbf{N}$ then we can use the substitution $x=e^{-i\theta} t$ which preserve the Mahler measure. We know that the limit of the Mahler measure $\textrm{M}(T_{n,a})$ exist as $n\to \infty$ \cite{Boy3}. Thus we can determine it using the subsequence $n=kq$. We obtain trinomial $t^{np}-|a|t-1$. Finally, we can use Theorem \ref{cha:Product} to calculate the limit.
\end{rem}

\section{The computation}

\begin{table}[!htbp]
\caption{$\frac{\nu_{n,a}}{n}$, $\textrm{M}(T_{n,a})$, $\exp(\frac{1}{2\pi}\int_0^{\arccos(-\frac{a}{2})}\ln(1+a^2+2a\cos t)dt )$} 
\centering 
\begin{tabular}{c c c c c c c} 
\hline\hline 
$a$ &	$\frac{\nu_{100,a}}{100}$ &	$\frac{\nu_{150,a}}{150}$ & $\frac{1}{\pi}\arccos\frac{-a}{2}$ & $\textrm{M}(T_{100,a})$ & $\textrm{M}(T_{150,a})$ & $\exp(\frac{1}{2\pi}\int\ldots$\\ [0.5ex] 
\hline 
0.1 & 0.51 & 0.51333  & 0.51592   & 1.0323479 & 1.0323491  & 1.0323476 \\
0.2 & 0.53 & 0.52667  & 0.53188   & 1.0657805 & 1.0657672  & 1.0657699 \\
0.3 & 0.55 & 0.55333  & 0.54793   & 1.1003457 & 1.1003221  & 1.1003332 \\
0.4 & 0.57 & 0.56667  & 0.56409   & 1.1360971 & 1.1361120  & 1.1361098 \\
0.5 & 0.57 & 0.58000  & 0.58043   & 1.1731391 & 1.1731919  & 1.1731790 \\
0.6 & 0.59 & 0.59333  & 0.59699   & 1.2116363 & 1.2116381  & 1.2116281 \\
0.7 & 0.61 & 0.60667  & 0.61382   & 1.2515943 & 1.2515324  & 1.2515544 \\
0.8 & 0.63 & 0.63333  & 0.63099   & 1.2931127 & 1.2930638  & 1.2930665 \\
0.9 & 0.65 & 0.64667  & 0.64858   & 1.3363117 & 1.3363123  & 1.3362872 \\
1.0 & 0.67 & 0.66000  & 0.66667   & 1.3813362 & 1.3813469  & 1.3813564 \\
1.1 & 0.69 & 0.68667  & 0.68537   & 1.4283607 & 1.4284371  & 1.4284355 \\
1.2 & 0.71 & 0.70000  & 0.70483   & 1.4775944 & 1.4777358  & 1.4777126 \\
1.3 & 0.73 & 0.72667  & 0.72523   & 1.5292849 & 1.5294002  & 1.5294116 \\
1.4 & 0.75 & 0.74000  & 0.74682   & 1.5837217 & 1.5838093  & 1.5838036 \\
1.5 & 0.77 & 0.76667  & 0.76995   & 1.6412364 & 1.6412643  & 1.6412260 \\
1.6 & 0.79 & 0.79333  & 0.79517   & 1.7022009 & 1.7021417  & 1.7021144 \\
1.7 & 0.81 & 0.82000  & 0.82340   & 1.7670223 & 1.7670956  & 1.7670601 \\
1.8 & 0.85 & 0.84667  & 0.85643   & 1.8370110 & 1.8369104  & 1.8369342 \\
1.9 & 0.89 & 0.90000  & 0.89892   & 1.9132863 & 1.9131780  & 1.9132259 \\
2.0 & 0.99 & 0.99333  & 1.00000   & 2.0000000 & 2.0000000  & 2.0000000 \\
\hline 
\end{tabular}
\label{table:nu} 
\end{table}

Finally we present in Table 1 the computing of the rate $\frac{\nu_{n,a}}{n}$ and the Mahler measure of the trinomial $T_{n,a}=x^n-ax-1$ for $n=100,150$, $a=0.1,0.2,\ldots,2.0$ as well as its limits. Generally, in the case $n=150$ results are much closer to the limit than in the case $n=100$. Nevertheless there are results which are much more distant from the limit in the case $n=150$, for example $a=0.1$ and $a=1.5$. If $a=1$ we get the smallest known limit point of nonreciprocal measures
$\lim_{n\to\infty} M(z^n - z - 1) = 1.38135\ldots$ (Boyd \cite{Boy2}). The last column in the table could be obtained in another way: using the exponential function of the right side of \eqref{Maillot}.

\textbf{Acknowledgements.} I am grateful to David Boyd, Yann Bugeaud, Arturas Dubickas, Kevin Hare, Aleksandar Ivi\'c, Gradimir Milovanovi\'c and Andrzej Schinzel for their comments, suggestions and corrections.


\begin{thebibliography}{14}
\normalsize 

\bibitem{AK} S. Akiyama, V. Komornik,
\textit{Discrete spectra and Pisot numbers},
J. Number Theory, \textbf{133}: (2)(2013) 375--390

\bibitem{Bor} P. Borwein, Computational Excursions in Analysis and Number Theory, CMS Books in Mathematics, Springer (2002). 

\bibitem{Ber} M.-J. Bertin, A. Decomps-Guilloux, M. Grandet-Hugot, M. Pathiaux-Delefosse, and J.-P. Schreiber, \textit{Pisot and Salem numbers}, Birkh\"{a}user, Basel, (1992).

\bibitem{Boy1} D. W. Boyd, The maximal modulus of an algebraic integer, Math. Comp. 45 (1985), 243--249. 

\bibitem{Boy2} D. W. Boyd, Variations on a theme of Kronecker, Canad. Math. Bull. 21 (1978), 129--133.

\bibitem{Boy3} D. W. Boyd, Speculations concerning the range of Mahler's measure. Canad. Math. Bull. 24 (1981), no. 4, 453 -- 469.

\bibitem{DU2} A. Dubickas, On numbers which are Mahler measures, Monatsh. Math. 141 (2004), 119--126. 

\bibitem{ET} P. Erd{\H o}s and P. Tur\'{a}n, On the distribution of roots of polynomials, Annals of Math. 51 (1950), 105--119

\bibitem{Leh} D. H. Lehmer, Factorization of certain cyclotomic functions, Ann. of Math. (2), v. 34, (1933), pp.461--479.

\bibitem{Lin1} D. Lind, Entropies and factorizations of topological Markov shifts, Bull. Amer. Math. Soc. (N.S.) 9 (1983), 219--222.

\bibitem{Ma} V. Maillot, G\'{e}om\'{e}trie d’Arakelov des vari\'{e}t\'{e}s toriques et ﬁbr\'{e}s en droites int\'{e}grables, M\'{e}m. Soc. Math. Fr. (N.S.) 80 (2000), 129 pp.

\bibitem{Mil} G.V. Milovanovi\'c, D.S. Mitrinovi\'c and Th.M. Rassias, Topics in Polynomials: Extremal Problems, Inequalities, Zeros. World Scientific Publ. Co., Singapore, New Jersey, London, (1994).

\bibitem{SC} A. Schinzel, A class of algebraic numbers, Tatra Mt. Math. Publ. 11 (1997), 35--42. 

\bibitem{Sel} E. Selmer, On the irreducibility of certain trinomials, Math. Scand. 4 (1956), 287--302

\bibitem{SS} N. Sidorov, B. Solomyak, On the topology of sums in powers of an algebraic number, Acta Arith. 149 (2011), 337--346

\bibitem{SM} C. J. Smyth, The mean values of totally real algebraic integers, Math. Comp. 42 (1984), 663--681. 

\bibitem{St1} D. Stankov, On spectra of neither Pisot nor Salem algebraic integers,
\textit{Monatsh. Math.}, \textbf{159} (2010), 115--131.

\bibitem{St2} D. Stankov, On linear combinations of Chebyshev polynomials, arXiv:1311.2230

\bibitem{Van} S. Vandervelde, A formula for the Mahler measure of axy+bx+cy+d, 
J. Number Theory, 100 (2003), pp. 184--202

\bibitem{Wu} Q. Wu, The smallest Perron numbers
Math. Comp. 79, (272), (2010), 2387--2394




















\end{thebibliography}

\end{document}